\documentclass[12pt]{amsart}
\usepackage{amsmath,amssymb,amsfonts,latexsym}
\usepackage{mathrsfs}%pour les caracteres en calligraphie
\usepackage[dvips]{geometry}
\usepackage{array}
\usepackage{epsf}
\usepackage{epsfig}
\newtheorem*{lemma}{Lemma}
\newtheorem*{prop}{Proposition}
\newtheorem*{thm}{Theorem}
\newtheorem*{cor}{Corollary}
\newcommand{\iso}{\overset{\sim}{\rightarrow}}

\newcommand{\twoheaddownarrow}{\overset{\sim}{\twoheaddownarrow}}

\newcommand{\nc}{\newcommand}
\nc{\Ker}{\operatorname{Ker}} \nc{\rker}{\operatorname{rKer}}
\nc{\im}{\operatorname{Im}}
\nc{\stab}{\operatorname {Stab}}
\nc{\ann}{\operatorname {Ann}}
\nc{\Id}{\operatorname {Id}}
\nc{\Prim}{\operatorname {Prim}}
\nc{\Real}{\operatorname {Re}}
\nc{\Ext}{\operatorname {Ext}}
\nc{\rad}{\operatorname {rad}}

\begin{document}

\title[Convexity for $S$-graphs]{Convexity Properties of the Canonical $S$-graphs}

%\title[The $B(\infty)$ Crystal]{Dual Kashiwara Functions for the $B(\infty)$ Crstal}
\author[Anthony Joseph]{Anthony Joseph}

\date{\today}
%\footnote{Work supported in part by the Binational Science Foundation, Grant no. 711628}
\maketitle

\vspace{-.9cm}\begin{center}
Donald Frey Professional Chair\\
Department of Mathematics\\
The Weizmann Institute of Science\\
Rehovot, 76100, Israel\\
anthony.joseph@weizmann.ac.il\footnote{Work supported in part by the Binational Science Foundation, Grant no. 711628}
\end{center}\

\medskip


\

\textbf{Abstract}.  Let $n$ be a positive integer and set $N=\{1,2,\ldots,n\}$. Let $\{c_k\}_{k \in N}$ be non-negative integers.  A convex set $(c_k')_{k \in N} \subset \mathbb Q^n$, given by a family of linear relations in the $\{c_k\}_{k \in N}$ and depending on their natural order, is defined. The extremal points of this convex set is shown to be the
$S$-set constructed in \cite {J1}. A main application of this result is towards \cite {J2} a precise description of Kashiwara $B(\infty)$ crystal.

\section{Introduction}\label{1}

We shall assume the base field to be the set of rational numbers $\mathbb Q$.  We could equally well replace $\mathbb Q$ by the real field $\mathbb R$.

\subsection{$S$-graphs}\label{1.1}

In \cite [6.7]{J1} $S$-graphs were introduced to understand the structure of the Kashiwara $B(\infty)$ crystal for an arbitrary Kac-Moody Lie algebra.  In this it was noted that $B(\infty)$ must have a polyhedral structure if ``dual Kashiwara functions'' associated to a given simple root, exist and are linear. These functions are not intrinsically canonical; but their maximal values are, and moreover the latter determine the required polyhedral structure.

When the Weyl group $W$ is finite then a result \cite [Thm. 3.9]{BZ} of Berenstein-Zelevinsky shows that one may compute dual Kashiwara functions though the tropical calculus.  This method breaks down in general.  The role of $S$-graphs was to give a procedure which should always work.

Even when the tropical calculus can be applied it gives (surprisingly) far too many functions (in the sense that only their maxima are to be used).  This ambiguity is absent from the $S$-graph method in so far as it is applicable.  Thus even when $W$ is finite, one should show that the tropical calculus gives the integral points of a convex set whose extremal elements are those obtained from $S$-graphs.

The main result of the present work describes a convexity property (Thm. \ref {1.4}) associated with $S$-graphs aimed at ultimately establishing the property alluded to in the previous paragraph.  It is yet one more property of these remarkable graphs.  We also give in Section \ref {2} a new method to read off a function from a tableau and in Section \ref {4} a method to reconstruct the tableau.

\subsection{Functions}\label{1.2}

Fix a positive integer $n$ and set $N:=\{1,2,\ldots,n\}, \hat{N}=\{1,2,\ldots, n+1\}$.  An $S$-graph $\mathscr G$ (of order $n$) has a finite vertex set $V(\mathscr G)$ labelled by elements of $\hat{N}$.  For each $k\in \hat{N}$, let $V^k(\mathscr G)$ denote the set of all vertices of $\mathscr G$ with label $k$. Then the crucial ``$S$'' property of an $S$-graph is for all $v\in V(\mathscr G)$ and all $k \in \hat{N}$, there exist $v' \in V^k(\mathscr G)$ and an ordered path from $v$ to $v'$.

The above property would lead to an immediate contradiction had we used arrows on edges to define ordering.  Instead we assign to each edge of $\mathscr G$ a non-negative integer $c_k: k \in N$ and require that the values of these integers increase along an ordered path \cite [6.7]{J1}.  It would seem to us that graphs with this $S$ property should be of wide interest.

A further important property of an $S$-graph is that it admits ``evaluation'' \cite [6.2]{J1}.  This means that to each $v \in V(\mathscr G)$ there is a linear function $z_v$ such that for each pair $v \in V^k(\mathscr G), v' \in V^\ell(\mathscr G)$ joined by an edge with label $c_r$ one has
$$z_v-z_{v'}=c_r(r^k-r^\ell),\eqno {(1)}$$
where the $r^i:i \in \hat{N}$ are linear functions on an $n+1$ dimensional vector space over $\mathbb Q$ arising from the Kashiwara functions on the $B(\infty)$ crystal (\ref {3.3}).

One may remark that an $S$-graph may admit cycles, so this condition is not trivially satisfied.

Here the $x_k:=r^k-r^{k+1}$ are viewed as co-ordinate functions on an $n$ dimensional vector space over $\mathbb Q$.

In addition to the above there is a distinguished vertex $v \in V(\mathscr G)$ to which we assign the ``driving function''.  Here for the most part it will be taken to be the zero function.  Assume further that $\mathscr G$ is connected. Then the set $z_v: v \in V(\mathscr G )$ of linear functions is determined by $(1)$ and the driving function. Such a set is called an $S$-set.  The $r^j: j \in \hat{N}$ above are identified as Kashiwara functions (\ref {3.3}) for a particular choice of simple root.  Then the $S$ property gives, via a sum rule - hence the epithet $S$, the maxima of the resulting ``dual Kashiwara functions'' a required invariance property \cite [6.7]{J1} needed for describing $B(\infty)$ as a polyhedral subset of $\mathbb N^\infty$.

%The set of functions obtained from an $S$-graph is called an $S$-set.

%The existence of such graphs for all $n \in \mathbb N^+$ is not obvious.  Indeed the reader might first wish to reflect how such graphs might be constructed.

\subsection{Order Relations}\label{1.3}

 The structure of an $S$ graph of order $n$ should of course depend on the natural order relation on the set $\{c_k\}_{k \in N}$.  Fix a linear order (also known as a total order) on $N$.  Since there is a danger that this order relation may be confused with the natural order on $N$ we note it by $\prec$.  (This precaution was not taken in \cite {J1}).

  The relation $\prec $ induces a linear order on $\{c_k\}_{k \in N}$. To simplify notation we use $\textbf{c}$ to denote this set equipped with $\prec$ which we shall always assume lifts the natural order on $\{c_k\}_{k \in N}$. In other words $i\prec j$ implies $c_i\leq c_j$.  We further let $<$ denote the natural linear order on $N$.

 Even imposing some further natural conditions (\cite [$P_1-P_7$]{J1} some of which are the above) an $S$-graph is not uniquely determined by $\textbf{c}$.  However it is shown in \cite {JZ} that there is a $S$-graph $\mathscr G(\textbf{c})$, uniquely determined by $\textbf{c}$, appearing as a subgraph of the graph $\mathscr G^{n+1}$ of links of equivalence classes $H^{n+1}$ of unordered partitions into $(n+1)$ parts satisfying certain boundary conditions.  It is these $S$-graphs which we consider.  It is also shown in \cite [5.8]{JL} that $Z(\textbf{c}):=\{z_v:v \in \mathscr G(\textbf{c})\}$ is independent of the lifting of the natural order on $\{c_k\}_{k \in N}$.

 \subsection{Convexity}\label{1.4}

  The canonical $S$-graphs $\mathscr G(\textbf{c})$ are constructed by a process which we call binary fusion \cite [7.2]{J1}.  We review this in Section \ref {3.3}.  For the moment we fix $s_i:i \in N$ such that $s_1\prec s_2 \prec \ldots \prec s_n$.  Then we relabel the set $N_k:=\{s_i\}_{i=1}^k$ as $\{t_i\}_{i=1}^k$, so that $t_1 < t_2 < \ldots < t_k$.

 In the above conventions, set $K(\textbf{c}):=(c_k')_{k \in N} \subset \mathbb Q^n$ defined by
 $$0\leq c_k' \leq c_k,  \eqno {(2)}$$
 %$$c_{k+1}'-c_k' \geq \min\{0,c_{k+1}-c_k\}, \forall k \in N\setminus \{n\}. \eqno {(3)}$$
 $$c_{t_{i+1}}'-c_{t_i}' \geq \min\{0,c_{t_{i+1}}-c_{t_i}\}, \forall i=1,2,\ldots,k-1, \eqno {(3)_k}$$
 for all $k \in N$.
 It is a convex set, Lemma \ref {3.6}.  Let $E(K(\textbf{c}))$ denote the extremal points of $K(\textbf{c})$.  The main result of this work is the following

 \begin {thm} $E(K(\textbf{c}))= Z(\textbf{c})$.
 \end {thm}

 \subsection{}\label{1.5}

 Recall the above notation.  Take $k \in N$. Set $K=\{1,2,\ldots,k\}$. Then $s_k=t_j$ for some $j \in K$.  By construction $c_{s_k}=c_{t_j} \geq c_{t_{j+1}},c_{t_{j-1}}$.  It follows that $(3)_k$ implies the relations
 $$c_{t_{j+1}}'-c_{t_j}' \geq -(c_{t_{j}}-c_{t_{j+1}}), \quad c_{t_{j}}'-c_{t_{j-1}}' \geq 0. \eqno{(3)'_k}$$

  In what follows relation $(3)$ (resp. $(3)'$ means the combination of $(3)_k:k \in N$ (resp. $(3)'_k:k \in N$).

  \begin {lemma}   $(3)$ and $(3)'$ are equivalent.
  \end {lemma}

  \begin {proof}  It remains to prove that $(3)'$ implies $(3)$.

  Suppose that for some $k \in N$ that $t>t'$ are nearest neighbours in $N_k$, with respect to $<$.  If either $t$ or $t'$ is the unique maximal element of $N_k$, with respect to $\prec$, then the condition on $c_t'-c'_{t'}$ given by $(3)_k$ is the same as that given by $(3)'_k$. Otherwise $t,t' \in N_{k-1}$ and are nearest neighbours in $N_{k-1}$.  Then the proof proceeds by the obvious induction.

  \end {proof}

 \textbf{Remark 1.}  If the coefficients $c_k$ are increasing in $k$, then $(3)_n$ just asserts that $c'_k$ are increasing in $k$.  In this case $(3)_n$ implies $(3)_k$, for all $k \in N$.  However this is generally false even for $n=3$ as was pointed out to me by S. Zelikson through the following example.  Suppose $1 \prec 3 \prec 2$.  Then $(3)_3$ give $c'_2-c_1' \geq 0, c'_3-c'_2 \geq -(c_2-c_3)$, which together imply $c'_3-c'_1 \geq -(c_2-c_3)$.  Yet $(3)_2$ gives the relation $c'_3-c'_1 \geq 0$, which is generally stronger.  In this case the above set of relations could have been summarized as
 $$c_{r}'-c_s' \geq \min\{0,c_{r}-c_s\}, \forall r,s \in N|r>s. \eqno {(3)''}$$

 \

\textbf{ Remark 2.}  Obviously $(3)''$ implies $(3)$ but the converse can be false even for $n=3$.
   Thus suppose $2 \prec 1 \prec 3$.  Then $(3)_3$ give $c'_2-c_1' \geq -(c_1-c_2), c'_3-c'_2 \geq 0$, which imply $(3)_2$.  Together they imply $c'_3-c_1' \geq -(c_1-c_2)$, which is weaker than we would have obtained from $c_{r}'-c_s' \geq \min\{0,c_{r}-c_s\}$, taking $r=3,s=1$, namely $c_3'-c'_1 \geq 0$.

 %In what follows relation $(3)$ (resp. $(3)'$ means the combination of $(3)_k:k \in N$ (resp. $(3)'_k:k \in N$).  Of course $(3)$ and $(3)'$ are equivalent.
 \subsection{}\label{1.6}

 The main results of this paper are applied to the description of the dual Kashiwara functions for the Kashiwara crystal $B(\infty)$ in \cite {J2}, for example in \cite [Lemmas 7.2.7, 8.5]{J2}.

 \

 \textbf{Acknowledgements.}  I should like to thank S. Zelikson for pointing out to me that $(3)_n$ was insufficient to describe the set $K(\textbf{c})$ satisfying Theorem \ref {1.4}.  His understanding of the results in \cite {BZ} also helped to inspire this work.

 \section{Tableaux}\label{2}

 This section should be skipped on a first reading as it requires some knowledge of the results in \cite {J1} which we shall not recall in great detail. Then a fairly detailed knowledge of \cite {J1} will be required.

\subsection{}\label{2.1}

Fix $n \in \mathbb N$. Let us first review the main properties of the set $H^{n+1}$ (or simply, $H$) of equivalence classes of diagrams having $n+1$ columns (and satisfying certain boundary conditions) given in \cite [2.1]{J1}.  The boundary conditions are basic to the structure of the sets of dual Kashiwara functions and in turn are mainly responsible for the lack of inversion symmetry $k \mapsto n+1-k$ in Equation $(3)$.

Let $\mathscr G_{H^{n+1}}$ (or simply, $\mathscr G^{n+1}$) be the graph whose vertices are the elements $H^{n+1}$ and whose edges are the single block linkages \cite [3.3,6.1]{J1}.

Recall:

 There is a ``complete'' diagram of minimal height \cite [2.3.7]{J1} in each equivalence class.

  To each diagram we associated a tableau $\mathscr T$ defined \cite [3.1]{J1} by inserting the coefficients $c_k: k \in N$ into blocks in a unique fashion (so then the notion of diagram and tableau become interchangeable). Through this one associates to each tableau $\mathscr T$ a function $f_\mathscr T$.  It is independent of the choice of $\mathscr T$ in its equivalence class \cite [4.5]{J1}. Conversely the map $\mathscr T \mapsto f_\mathscr T$, separates these classes \cite [4.7]{J1}.  When the fixed ``driving function'' is set equal to zero, $f_\mathscr T$ takes the form
  $$f_\mathscr T=\sum_{i\in N}c_i^{\mathscr T}(r^i-r^{i+1}), \eqno {(4)}$$
  where the $c_i^{\mathscr T}$ are linear combinations with integer coefficients of the $\{c_i\}_{i\in N}$.

  For each choice of $\textbf{c}$ there is a subset $H^{n+1}(\textbf{c})$, or simply $H(\textbf{c})$, of $H^{n+1}$ of tableaux determined by rules given in \cite [3.2]{J1}.   The subgraph of $\mathscr G^{n+1}$ whose vertices lie in $H(\textbf{c})$ is denoted by $\mathscr G(\textbf{c})$.

  \

 \textbf{ Notation.}  When $\mathscr T \in H(\textbf{c})$  defines a vertex $v$ of $\mathscr G(\textbf{c})$ we shall denote $f_\mathscr T$ by $z_v$.

 \

  Suppose that $\mathscr T$ is a complete tableau.  It was shown in \cite [5.2]{J1} how to compute $f_\mathscr T$ from the heights of its columns.  Here we give a new formula which is perhaps better and in any case allows us to show that $K(\textbf{c})\supset \{z_v\}_{v \in V(\mathscr G(\textbf{c}))}$.  Conversely (Section \ref {4}) it eventually allows one to reconstruct a tableau from a function.

\subsection{}\label{2.2}

The main result of this section is to show that how to obtain the coefficients $c_i^\mathscr T$ from $\mathscr T \in H^{n+1}$.  In view of $(4)$ this allows one to read off the function $f_\mathscr T$ from $\mathscr T$.  It provides a useful alternative to the method described in \cite [Sect. 4]{J1}.

  The set of columns of a tableau $\mathscr T \in H^{n+1}$ is denoted by $\{C_i\}_{i\in \hat{N}}$. The height of a column $C$ is denoted by $ht(C)$. The height function of $\mathscr T$ is the map $i\mapsto ht(C_i)$. Set $ht(\mathscr T)=\max \{ht (C_i)\}_{i\in \hat{N}}$.

  Let $\mathscr T$ be a complete tableau. In this case the height function of $\mathscr T$ takes a rather special form \cite [2.3.3]{J1}.  Again the numbering the blocks of $\mathscr T$ is also rather special \cite [3.1]{J1}. Indeed a block lying in an odd (resp. even) row of the $k^{th}$ column $C_k$ has entry $c_k$ (resp. $c_{k-1}$), except if $k=n+1$ (resp. $k=1$), in which case (if it exists) it is called an extremal block \cite [3.1]{J1} and is not given a numerical entry.

  In a row of odd (resp. even) height $m$ the contribution to $f_\mathscr T$ from $C_k$ (resp. $C_{k+1}$) is $c_k(r^k-r^{j})$ (resp. $c_k(r^{k+1}-r^{j})$), where $j>k$ (resp. $j<k+1$) is minimal (resp. maximal) such that $C_j$ has height $\geq m$.  The boundary conditions ensure the existence of such a $j$ for a complete tableau \cite [Lemma 2.3.2]{J1}.  Again this formula does not need there to be a numerical entry in an extremal block.

  Recall \cite [2.2]{J1} that distinct columns $C,C'$ of $\mathscr T$ of height $\geq s$ are said to be neighbouring at level $s$ if every column of $\mathscr T$ between $C,C'$ has height $<s$.

  The left (resp. right) indicator $\iota^-_{C}$ (resp. $\iota^+_{C}$) of a column $C$ is set equal to $0$ if there is no column of height $\geq ht(C)$ to the left (resp. right) of $C$ and set equal to $1$ otherwise.

  Set $c_{n+1}=0$.  Given integers $m\leq n$, set $[m,n]:=\{m,m+1,\ldots,n\}$.

\begin {lemma}  Assume $k \in [0,n]$. Let $C_{k+1}$ have height $m \in \mathbb N$.  Then

\

(i) Suppose $m$ is odd.  If there exists $j \in [1,k]$ such that $C_j$ and $C_{k+1}$ are neighbouring at level $m$, then $c_{k+1}^{\mathscr T}-c^{\mathscr T}_k=c_{k+1}-c_j$.  Otherwise  $c_{k+1}^{\mathscr T}-c^{\mathscr T}_k=c_{k+1}$.  Expressed less pedantically $c_{k+1}^{\mathscr T}-c^{\mathscr T}_k=c_{k+1}-\iota^-_{C_{k+1}}c_j$.

%(i)   If $m$ is odd, then $c_{k+1}^{\mathscr T}-c^{\mathscr T}_k=c_{k+1}-\iota^-_{C_{k+1}}c_j$, where $j\leq k$ is maximal such that $C_j$ has height $\geq m$ (so then $C_j$ and $C_{k+1}$ are neighbouring at level $m$).
%Moreover $v \in V(\mathscr G(\textbf{c}))$ implies that $\iota^-_{C_{k+1}}c_j \leq c_k$.  In particular $c_{k+1}^v-c^v_k\geq c_{k+1}-c_k$.

\

(ii)  Suppose $m$ is even and strictly positive.  If there exists $j\in [k+1,n]$ such that $C_{k+1}$ and $C_{j+1}$ are neighbouring at level $m$, then $c_{k+1}^{\mathscr T}-c^{\mathscr T}_k=c_{k+1}-c_j$. Otherwise $c_{k+1}^{\mathscr T}-c^{\mathscr T}_k=c_{k+1}$.  Expressed less pedantically $c_{k+1}^{\mathscr T}-c^{\mathscr T}_k=c_{k+1}-\iota^+_{C_{k+1}}c_j$.

%then $c_{k+1}^{\mathscr T}-c^{\mathscr T}_k=c_{k+1}-\iota^+_{C_{k+1}}c_j$, where $j\geq k+1$ is minimal such that $C_{j+1}$ has height $\geq m$ (so then $C_{k+1}$ and $C_{j+1}$ are neighbouring at level $m$).
%Moreover $v \in V(\mathscr G(\textbf{c}))$ implies that $\iota^+_{C_{k+1}}c_j \leq c_{k+1}$.  In particular $c_{k+1}^v-c^v_k\geq 0$.

\

(iii) If $m=0$, then $c_{k+1}^{\mathscr T}-c^{\mathscr T}_k=0$.

%\
%
%(iv)  In all cases $c_{k+1}^v-c^v_k\geq \min \{0,c_{k+1}-c_k\}$.

%(iv)  In all cases $c_{k+1}^v-c_k^v\geq \min\{0,c_{k+1}-c_{k}\}$.

%\
%
%
% (v)
% $$c_n^v=\left\{
%                       \begin{array}{ll}
%                         c_n, & \hbox{if $C_{n+1}$ has odd height,} \\
%                         0, & \hbox{otherwise.}
%                       \end{array}
%                     \right. $$

 \end {lemma}

 \begin {proof}  We shall compute $f_\mathscr T$ using \cite [4.4]{J1}.

 %and equate with the right hand side of $(18)$.  In this the driving function denoted by $h$ (resp. $z$) here (resp. in \cite {J3}) can be ignored as it is common to both expressions.  We do not need the complicated expressions given in \cite [5.2]{J3}. In this sense the computation is novel.

 Strictly speaking the proof below is only valid when $k>0$.
  A few extra computations are needed for the case $k=0$.
   %and these use \cite [Lemma 2.3.2]{J3} to determine the height of $C_1$.  Alternatively
   Alternatively it may be validated for the case $k=0$ by adjoining an extra empty column $C_0$ to $\mathscr T$ on the left.  Notably the new tableau $\mathscr T_0$ so obtained still satisfies the boundary conditions \cite [2.2]{J1}. The rules \cite [4.4]{J1} for computing $f_{\mathscr T_0}$ show that the result is independent of $c_0$ which can be set equal to zero. Again since the first column is empty, there are no contributions to $r^0-r^1$ in $f_{\mathscr T_0}$ and so $c_0^\mathscr T=0$.  Finally $\mathscr T_0$ may be completed.  This just amounts to adjoining a suitable number of dominoes to the zeroth column \cite [2.3.1]{J1}.

   \

 For (i), suppose $m$ is odd.

 \

  Consider the contribution to $r^k-r^{k+1}$ (resp. $r^{k+1}-r^{k+2}$) from the rows of odd height $m+2s:s \in \mathbb Z$.  If $s \geq 1$ these contributions come from a column $C_\ell$ of height $ \geq m+2s$ with $\ell$ maximal such $1\leq \ell\leq k$.  Thus these contributions are the same.  If $s=0$, this contribution is $\iota^-_{C_{k+1}}c_j$ (resp. $c_{k+1}$) with $j$ defined as in (i) of the lemma.  On the other hand the height function of a complete tableau \cite [2.3.3]{J1} implies that $C_k$ has height at least $m-2$.  Thus the contribution for $s\leq -1$ is $c_k$ (resp. $c_{k+1}$).   Thus the overall contribution to $c_{k+1}^\mathscr T-c_k^\mathscr T$ coming from rows of odd height is just $c_{k+1}-\iota^-_{C_{k+1}}c_j + \frac{m-1}{2}(c_{k+1}-c_k)$.

 Consider the contribution to $r^k-r^{k+1}$ (resp. $r^{k+1}-r^{k+2}$) from the rows of even height $m+2s-1:s \in \mathbb Z$.  If $s \geq 1$, these contributions come from a column $C_\ell$ of height $ \geq m+2s-1$ with $\ell$ minimal such $n+1 \geq  \ell\geq k+2$.  Thus these contributions are the same. On the other hand the height function of a complete tableau \cite [2.3.3]{J1} implies that $C_{k+2}$ has height at least $m-1$.  Then for $s \leq 0$ the contribution is $-c_k$ (resp. $-c_{k+1}$).  Thus the overall contribution to $c_{k+1}^\mathscr T-c_k^\mathscr T$ coming from rows of even height is just $- \frac{m-1}{2}(c_{k+1}-c_k)$.  Combined with the expression in the previous paragraph, this gives the first part of (i).

 %Finally suppose that $\iota^-_{C_{k+1}}=1$, equivalently that the column $C_j$ is defined.  Then its entry in its $m^{th}$ row is just $c_j$. On the other hand the entry in the $m-1$ row of $C_{k+1}$ is $c_k$. By the rules given in  \cite [3.2.1$(**)$,5.3]{J3} the condition that $v \in V(\mathscr G(\textbf{c}))$ implies that $c_j \leq c_k$.  Hence the last two parts of (i).

 \

 For (ii) suppose that $m$ is even. The argument is similar so we can be a little briefer.

 \

  Consider the contribution to $r^{k}-r^{k+1}$ (resp. $r^{k+1}-r^{k+2}$) from the rows of odd height $m+2s-1:s \in \mathbb Z$.  If $s \geq 1$ these contributions are the same.  On the other hand the height function of a complete tableau \cite [2.3.3]{J1} implies that $C_k$ has height at least $m-1$.  Thus the contribution for $s\leq 0$ is $c_k$ (resp. $c_{k+1}$).   Thus the overall contribution to $c_{k+1}^\mathscr T-c_k^\mathscr T$ coming from rows of odd height is just $\frac{m}{2}(c_{k+1}-c_k)$.

 Consider the contribution to $r^{k}-r^{k+1}$ (resp. $r^{k+1}-r^{k+2}$) from the rows of even height $m+2s:s \in \mathbb Z$.  If $s \geq 1$, these contributions are the same.  If $s=0$, this contribution is $-c_k$ coming from the column $C_{k+1}$ (resp. $-\iota^+_{C_{k+1}}c_j$ coming from the column $C_{j+1}$). On the other hand the height function of a complete tableau \cite [2.3.3]{J1} implies that $C_{k+2}$ has height at least $m-2$.  Then for $s \leq -1$ the contribution is $-c_k$ (resp. $-c_{k+1}$).  Thus the overall contribution to $c_{k+1}^\mathscr T-c_k^\mathscr T$ coming from rows of even height is just $-\iota^+_{C_{k+1}}c_j+c_k- \frac{m-2}{2}(c_{k+1}-c_k)$.  Together with the expression in the previous paragraph gives the first part of (ii).

 \

 %Finally suppose that $\iota^+_{C_{k+1}}=1$, equivalently that the column $C_{j+1}$ is defined.  Then its entry in the $m^{th}$ row is just $c_j$. On the other hand the entry in the $m-1$ row of $C_{k+1}$ is $c_{k+1}$. By the rules given in  \cite [3.2.1$(*)$,5.3]{J3} the condition that $v \in V(\mathscr G(\textbf{c}))$ implies that $c_j \leq c_{k+1}$.  Hence the last two parts of (ii).

If $m=0$, the only thing that changes in the proof of (ii) is that there is zero contribution to $c_{k+1}^\mathscr T-c_k^\mathscr T$ when $s=0$. Hence (iii).

%(iv) follows from (i)-(iii).

%\textbf{Remark}.  Since by definition of $f_\mathscr T$ one has $c^\mathscr T_{n+1}=0$, this result enables one to compute the $\{c_i^\mathscr T\}_{i=1}^n$ from the height function of $\mathscr T$.

%(v) can be checked directly but this is a little tricky since one must use that $\mathscr T$ is a complete tableau. It can also be read off from \cite [5.2,$(*)-(***)$]{J3}, taking account of \cite [Lemma 5.1]{J3} which implies that there are no cancellations between these expressions. Below we give the details. In this $i_j,j_{k_\ell}$ refer to the symbols used in \cite {J3}.  There is a slight clash of notation with that used here ($n$ appears in different roles).
%
% Let $u$ be the height of $\mathscr T$.  Suppose that $u$ is odd. Then $C_{n+1}$ has height $u$, since $\mathscr T$ is a complete tableau \cite [Lemma 2.3.2(ii)]{J3}.  Then the required contribution comes from \cite [5.2,$(***)$]{J3}, noting that $i_{m-1}<i_m=n+1$. Suppose that $u$ is even.  Then $C_{n+1}$ has height $\geq u-1$, since $\mathscr T$ is a complete tableau \cite [Lemma 2.3.2(i)]{J3}.  If equality holds then the required contribution comes from \cite [5.2,$(**)$]{J3}, noting that we can take $j_{k_n}<j_{k_{n+1}}=n+1$.  If the inequality holds then there is no contribution in view of \cite [5.2,$(*)$]{J3}, noting that $j_{k_{n+1}}<n+1$.

 \end {proof}

 \subsection{}\label{2.3}

 Let $\mathscr T$ be a tableau. In \cite [3.2.1]{J1}, we defined a partial order on $P(\mathscr T)$ on $N$.  It is \textit{not} independent of the choice of $\mathscr T$ in its equivalence class, except for the subset of equivalent complete tableaux \cite [Lemma 5.3]{J1}.

  %It has the property \cite [Lemma 5.7]{J1} that when $P(\mathscr T)$ lifts the order on $N$ defined by the natural order on the coefficients, then
 %$$0\leq c_i^\mathscr T \leq c_i, \forall i \in N.\eqno {(4)}$$

%In other words Equation $(1)$ is satisfied with $c_i'=c_i^\mathscr T$.

Now recall that we have used $\textbf{c}$ to define the set of coefficients combined with a linear order $\prec$ on $N$ lifting the natural order on the $\{c_k\}_{k \in N}$. Following \cite [Definition 5.3]{J1} we let $H(\textbf{c})$ denote the subset of $H$ of equivalence classes of tableaux such that $P(\mathscr T)$, with $\mathscr T$ complete, lifts to the given linear order $\prec$ on $N$.  Let $\mathscr G(\textbf{c})$ be the corresponding subgraph of $\mathscr G$.

This means that $\mathscr G(\textbf{c})$ depends on the partial order on the $\{c_k\}_{k \in N}$. It also depends on the lifting; but this can be avoided by identifying vertices corresponding to the same function.  The details are worked out in \cite [5.8]{JL}.  One may remark that because $P(\mathscr T)$ is only a partial order means that some of the graphs $\mathscr G(\textbf{c})$, for the different $n!$ orderings, may coincide.  The total number of such graphs is the Catalan number $C(n)$, \cite [6.7]{JL}.

Now suppose $\mathscr T \in H(\textbf{c})$.  Then by \cite [Lemma 5.7]{J1} we obtain
$$0\leq c_i^\mathscr T \leq c_i, \forall i \in N.\eqno {(5)}$$

Recall \cite [3.2.1]{J1} that $P(\mathscr T)$ can be described as follows. For simplicity we shall assume that $\mathscr T$ is complete which is all we need here.  This assumption simplifies the numbering in the blocks.

Assume $m$ odd.

 Let $C_j,C_{k+1}:j\leq k$ be columns of a complete tableau $\mathscr T$ of height $\geq m$ which are neighbouring at level $m$.  This means that the columns strictly between them have height $<m$.  Let $C_{j_1+1},C_{j_2+1},\ldots,C_{j_{s-1}+1}$ be the columns of $\mathscr T$ between $C_j,C_{k+1}:j\leq k$ having height $m-1$ and set $j_s=k$.  The entries in their $(m-1)$ row are respectively $c_{j_t}:t=1,2,\ldots,s$.  We remark that since $\mathscr T$ is complete, then if $s>1$ one has  $j_1=j$ by \cite [Lemma 2.3.3 (iii)]{J1}.  Then by \cite [3.2.1$(**)$]{J1} the relations $j\prec j_t:t=1,2,\ldots,s$ (and in particular $j\prec k$) belong to $P(\mathscr T)$.

Now fix a complete tableau $\mathscr T$ and let $\{C_i\}_{i \in \hat{N}}$ be its set of columns.  Suppose that $\iota^-_{C_{k+1}}=1$.  This means that there exists $j:1\leq j \leq k$ such that $C_j$ and $C_{k+1}$ are neighbouring at level $m$.  Then by the above $j\prec k$ belongs to $P(\mathscr T)$. Thus if $\mathscr T \in H(\textbf{c})$, we obtain $c_j \leq c_k$ and so from the conclusion of Lemma \ref {2.2}(i), that
$$c_{k+1}^{\mathscr T}-c^{\mathscr T}_k\geq c_{k+1}-c_k. \eqno {(6)}$$

Since $c_k \geq 0$, this again holds if $\iota^-_{C_{k+1}}=0$.

Suppose that $m$ is even.

Let $C_{k+1},C_{j+1}:j>k$ be columns of a complete tableau $\mathscr T$ of height $m$ and neighbouring at level $m$.  Then similar to the case when $m$ is odd (or by duality \cite [2.3.8]{J1}) we may deduce from \cite [3.2.1$(*)$]{J1} that the relation $j \prec k+1$ belongs to $P(\mathscr T)$.  Thus if $\mathscr T \in H(\textbf{c})$ we obtain $c_{j}\leq c_{k+1}$. Since $c_{k+1}\geq 0$, from the conclusion of Lemma \ref {2.2}(ii),(iii), we obtain
$$c_{k+1}^{\mathscr T}-c^{\mathscr T}_k\geq 0. \eqno {(7)}$$

%Combining $(6),(7)$ we obtain
%$$c_{k+1}^{\mathscr T}-c^{\mathscr T}_k\geq c_{k+1}-c_k. \eqno {(6)}$$

Comparison of $(5),(6),(7)$ above with $(2),(3)$ gives the
\begin {lemma}  Suppose $\mathscr T \in H(\textbf{c})$, then $f_\mathscr T=(c^\mathscr T_k)_{k \in N} $, satisfy $(2),(3)_n$ with $c_k'=c_k^\mathscr T: k \in N$.
\end {lemma}

\section{The Consequences of Binary Fusion}\label{3}

\subsection{}\label{3.1}

By \cite [Thm. 8.5]{J1} one may construct $\mathscr G(\textbf{c})$ by a method quite different to that described in \ref {2}.  We called this binary fusion \cite [7.2]{J1}.  It implies for example that $|H^{n+1}(\textbf{c})|=2^n$, which in general is quite impossible to prove directly, though the reader is invited to try.  (Lamprou checked this for $n=4$ in all $4!$ cases (and also that $|H_{n+1}|=42$) through a long and tedious computation.  This was a main inspiration for binary fusion.)  Binary fusion is also needed to show that $\mathscr G(\textbf{c})$ is an $S$-graph \cite [Sect. 7]{J1}.  It seems quite impossible to prove this from our previous construction; but again the reader is invited to try.

We recall briefly how binary fusion works.

Recall that $\textbf{c}$ designates the set of coefficients $\{c_k\}_{k \in N}$ with a fixed linear order $\prec$ on $N$ lifting the natural partial order on these coefficients.  Let $s$ denote the unique maximal element of $N$ with respect to $\prec$.  Recall this means that $c_s \geq c_k$, for all $k \in N$.

   Set $\textbf{c}^-=\textbf{c}\setminus c_s$. Assume that the labelled graph $\mathscr G(\textbf{c}^-)$ has been constructed, where the index set $N\setminus \{s\}$ (resp. $\hat{N}\setminus \{s\}$ is viewed as $\{1,2,\ldots,n-1\}$ (resp. $\{1,2,\ldots,n\}$) by closing up gaps in the obvious fashion (see \cite [7.2(1)]{J1}).  Define new graphs $\mathscr G^\pm$ isomorphic to $\mathscr G(\textbf{c}^-)$ as \textit{unlabelled} graphs. Let $\mathscr G^+$ be the graph obtained from $\mathscr G(\textbf{c}^-)$ by leaving the labels in $[1,s-1]$ unchanged and increasing the labels in $[s,n]$ by $1$. Let the labelling on $\mathscr G^-$ be defined by required that the hitherto defined unlabelled graph isomorphism $\varphi:\mathscr G^+ \iso \mathscr G^-$ fixes all labels with the following exception. If $v \in V(\mathscr G^+)$ has label $s+1$, then $\varphi(v) \in \mathscr G^-$ is assigned the label $s$.

  Then $\mathscr G(\textbf{c})$ is defined as the union of $\mathscr G^+$ and $\mathscr G^-$ in which each vertex $v$ of $\mathscr G^+$ with label $s+1$ is joined to $\varphi(v)$ with an edge having label $s$.

  It is convenient to view a function $f$ obtained from $\mathscr G(\textbf{c}^-)$ \textit{without} closing up gaps.  In particular $c_s$ and the $s^{th}$ co-ordinate function (represented by $r^s-r^{s+1}$ in $(4)$) will not appear in $f$.  Then $f$ may be expressed in the form
  $(c'_n,\ldots,c'_{s+1},0,c'_{s-1},\ldots,c'_1)$,
   where the $c'_i:i \in N \setminus \{s\}$ are certain linear combinations of the $c_i:i \in N \setminus \{s\}$.  Through binary fusion we obtain two elements $f^\pm$ arising from $\mathscr G(\textbf{c})$, that is  $f^+$ from $\mathscr G^+$ and $f^-$ from $\mathscr G^-$.

  \begin {lemma}  $f^\pm$ differ from $f$ only in coefficient of the $s^{th}$ co-ordinate function.
  \end {lemma}

  \begin {proof}   In our construction the functions obtained from $\mathscr G^+$ are those obtained from $\mathscr G(\textbf{c}^-)$ on replacing the term $c'_{s-1}(r^{s-1}-r^s)$ by $c'_{s-1}(r^{s-1}-r^{s+1})$ and on replacing $c_i'$ by $c'_{i+1}$, when $n-1\geq i \geq s$.   In terms of co-ordinates this means that $f$ becomes  $\tilde{f}^+:=(c'_{n},c'_{n-1}, c'_{s+1},c'_{s-1},c'_{s-1},c'_{s-2}\ldots, c'_1)$.

   On the other hand $\mathscr G^+$ is a subgraph of $\mathscr G(\textbf{c})$, so by \cite [Sect. 6.2, $(P_3)$]{J1} (reproduced here as eq. $(4)$) functions corresponding to adjacent vertices are related in the same manner.  Since $\mathscr G(\textbf{c})$ is connected by construction it is enough to show that there is just one vertex which describe the same function in each graph to conclude that this holds for every common vertex.

    We claim that there is a distinguished vertex common to both graphs to which we assign the same function (the ``driving function'' which we may take to be the zero function).  This distinguished vertex is the unique vertex with label $n+1$ in the unique pointed chain \cite [Sect. 6.3, $(P_4)$]{J1} lying in $\mathscr G(\textbf{c})$.  Now this chain has a unique edge with label $c_s$ which by the above construction connects a vertex with label $s$ in $\mathscr G^-$ to the corresponding vertex with label $s+1$ in $\mathscr G^+$.  It follows that the unique vertex in the pointed chain with label $i$ lies in $\mathscr G^-$ (resp. $\mathscr G^+$) if $i\leq s$ (resp. $i>s$).  Since $s \leq n$ this proves our claim.

    Thus $f^+=\tilde{f}^+$ which indeed differs from $f$ only in the $s^{th}$ co-ordinate.

   Finally the assertion for $f^-$ obtains from that for $f^+$ through  \cite [7.8(i),(ii)]{J1} (of whose conclusion is reproduced in $(8)$ below).

  \end {proof}

  \subsection{}\label{3.2}

  Some further properties of binary fusion are given below.

  Recall \cite [7.8(i),(ii)]{J1}, that for every vertex $v$ of $\mathscr G^+$ there exists $c(v) \in \textbf{c}^-\cup \{0\}$ such that
  $$z_{\varphi(v)}-z_v=(c_s-c(v))(r^s-r^{s+1}). \eqno {(8)}$$

  Here we remark that $c_s-c(v) \geq 0$, the equality being strict if $c_s > c_j$, for all $j \in N\setminus \{s\}$.

    Recall that $s$ (resp. $s+1$) is the missing label on the vertices of $\mathscr G^+$ (resp. $\mathscr G^-$).  Define a linear map $\theta: \sum_{j\in \hat{N}\setminus \{s\}}\mathbb Z r^j \rightarrow \sum_{j\in \hat{N}\setminus \{s+1\}}\mathbb Z r^j$, by $\theta(r^{s+1})=r^s$ and $\theta(r^j)=r^j$, for $j \neq s+1$.

   Through the definition of $\varphi$ and the formula $(1)$ relating $z_v,z_{v'}$  when $v,v'$ are neighbours in an $S$ graph $\mathscr G$, it follows that
    $$z_{\varphi(v)}-z_{\varphi(v')}=\theta(z_v)-\theta(z_{v'}). \eqno {(9)}$$

    \begin {prop} Suppose $\mathscr T \in H(\textbf{c})$, then $f_\mathscr T=(c^\mathscr T_k)_{k \in N} \in K(\textbf{c}) $.
    \end {prop}

    \begin {proof}  By Lemma \ref {2.3} it remains to show that $(3)_k$, holds for all $k \in N$.  The proof is by reverse induction on $k$.  For $k=n$, it results from Lemma \ref {2.3}.  For $k<n$ we may apply the induction hypothesis to the graph $\mathscr G(\textbf{c}^-)$.  Then the assertion follows from  Lemma \ref {3.1}.

    \end {proof}

    %\
%
%   \textbf{ Example.}  Take $n=3$ with $1 \prec 3 \prec 2$.  Let $r^{i+1}-r^i:i=1,2,3$ be represented by the $i^{th}$ co-ordinate function.  Then $\mathscr G(\textbf{c})=\{(0,0,0),(0,c_2-c_3,0);(0,c_2,c_3),(0,0,c_3);\newline (c_1,c_1,c_1),(c_1,c_1+c_2-c_3,c_1);(c_1,
%   c_1,c_3),(c_1,c_2,c_3)\}$

     \subsection{}\label{3.3}

      Let $B$ be the vector space $\mathbb Q^\infty$, that is to say consisting of countably many copies of $\mathbb Q$ in which all but finitely many entries are zero.  The Kashiwara functions $r^j_i: k \in \mathbb N^+$, where $i$ runs over an index set of simple roots are the linear functions on $B$ defined in \cite [2.3.2]{J0}.  They arise from the Kashiwara tensor product rule.  Here $i$ is fixed and this index suppressed and $\mathbb N^+$ is replaced by the finite subset $\hat{N}$.  For our present purposes the $r^j:j \in \hat{N}$ referred to in earlier sections as the Kashiwara functions may be viewed as infinite linear combinations of the co-ordinate functions on $B$.  Moreover the entries of $B$ corresponding to an index $>n+1$ may be taken to be equal to zero, so that the Kashiwara functions may be more simply viewed as linear functions on an $n+1$ dimensional vector space, which we again denote by $B$, over $\mathbb Q$. They are linearly independent and so for any $n+1$-tuple $(q_1,q_2,\ldots,q_{n+1}$ of elements of $\mathbb Q$, there exists $b \in B$ such that $r^j(b)=q_j$, for all $j \in \hat{N}$.

     \begin {lemma}  Take \textit{any} linear ordering on $Z(\textbf{c})$ and let $z_{\max}$ be the unique maximal element in $Z(\textbf{c})$ with respect to that ordering.  Then there is choice of $b \in B$ such that $z_{\max}(b)>z(b)$ for all $z \in Z(\textbf{c})\setminus\{z_{\max}\}$.
  \end {lemma}

  \begin {proof}

  The proof is by induction on $n$.  The assertion is obvious if $n=0$, since $\mathscr G(\textbf{c})$ is reduced to a single vertex.

  Recall the construction of binary fusion and the choice of $s$ given in \ref {3.1}.

  We may regard the linear ordering on $Z(\textbf{c})$ as a linear ordering on $\mathscr G(\textbf{c})$.  It induces a linear ordering on the subgraph $\mathscr G^+$.

  Now let $v^+_{\max}$ be the unique maximal element of $\mathscr G^+$.  By the induction hypothesis we can suppose that there exists  $b \in B$ such that $z_{v^+_{\max}}(b) > z(b)$ for all $z \in \mathscr G^+ \setminus \{z_{v^+_{\max}}\}$. Let $N^-, N^+$ be the maximal subsets of $\hat{N}\setminus \{s,s+1\}$ such that
  $$r^i(b)<r^{s+1}(b)<r^j(b), \forall i \in N^-, j \in N^+. \eqno{(10)}$$

   Set ${v^-_{\max}}=\varphi({v^+_{\max}})$. Since $s$ is the missing label on the vertices of $\mathscr G^+$, we can take $r^s(b)$ to be any element of $\mathbb Q$ and in particular such that $$r^i(b)<r^{s}(b)<r^j(b), \forall i \in N^-, j \in N^+, \eqno {(11)}$$

   Then by $(9)$ it follows that $z_{v^-_{\max}}(b) > \varphi(z)(b)$ for all $z \in \mathscr G^+ \setminus \{z_{v^+_{\max}}\}$.
  %and that the order relations between the $r^i(b):i \in N^-$ and the $r^j(b):j \in N^+$ have been left unchanged.

  Now let $v_{\max}$ be the unique maximal element of $\mathscr G(\textbf{c})$.  Suppose $v_{\max}$ belongs to $\mathscr G^+$ (resp. $\mathscr G^-$), that is $v_{\max}=v^+_{\max}$ (resp. $v_{\max}=v^-_{\max}$).   Clearly we can still choose $b \in B$ such that $r^s(b)<r^{s+1}(b)$ (resp. $r^s(b)>r^{s+1}(b)$) leaving the values of the $r^j(b):b \in \hat{N}\setminus \{s\}$ and the order relations in $(10),(11)$
  %(and those between the $r^i(b):i \in N^-$, $r^j(b):j \in N^+$)
  unchanged.

  It follows from $(8)$ that the choice of $b \in B$ satisfies the conclusion of the lemma, noting that $c_s-c(v^+_{\max})=0$, means that $z_{v^+_{\max}}=z_{v^-_{\max}}=z_{v_{\max}}$.

  \end {proof}

  %\textbf{Remark 1}.  This is one further remarkable property of the $S$-graphs described in \cite {JL}.

\textbf{Remark.}  Suppose that $b \in B$ is chosen so that the $r^j(b):j \in \hat{N}$ are pairwise distinct.  It is clear from the above proof that it is only the ordering between the $r^j(b):j \in \hat{N}$, which determines the ordering between the $z(b):z\in Z(\textbf{c})$.

 \subsection{}\label{3.4}

 %Recall (\ref {4.6}) that by definition $z_v =\theta(K_v)$ and that lemma \ref {4.6} asserts that $K_v \in \mathscr K^{BZ}_t$, for all $v \in V(\mathscr G(\textbf{c}))$.
% %Set $Z(\textbf{c})= \{z_v\}_ {v\in V(\mathscr G(\textbf{c}))}$.

 \begin {cor}  $Z(\textbf{c})=E(Z(\textbf{c}))$.
 \end {cor}

 \begin {proof}  Take $z \in Z(\textbf{c})$.  Suppose that there exists a finite subset $F$ of $Z(\textbf{c})$ and positive rationals $a^v:v \in F$ summing to $1$ such that $z=\sum_{v\in F}a^vz_v$.  By Lemma \ref {3.3}, there exists $b \in B$ such that $z(b)\geq z_v(b): v \in F$, with a strict inequality if $z\neq z_v$. Then substitution forces  $z(b)= z_v(b)$, for all $v \in F$.  Hence the assertion.
 \end {proof}

% \subsection{}\label{5.4}
%
% In order to obtain the reverse inclusion in Lemma \ref {4.10} we first need a converse to Lemma \ref {4.7}, that is to say there exists a complete tableau $\mathscr T \in H(\text{c})$ such that $z_v=\mathscr T$, satisfies the conclusion of Lemma \ref {4.7}.  For this we recall some further results of \cite {J3}.

 \subsection{}\label{3.5}

 Recall the notation and definitions of \ref {1.4}.

 \begin {prop}  $K(\textbf{c})$ lies in the convex hull of $Z(\textbf{c})$.
 \end {prop}

% In this section we show that the opposite inclusion to that of Lemma \ref {4.10} holds if $\text {(i)}'$ of \ref {4.6} holds.  In this we fix $s \in I$ and consider an element $K' \in \mathscr K_t$ obtained from an element $K\in \mathscr K_t$ by adjoining faces of type $s$ to $K$, that is $K'=K+\sum_{j=1}^n c_j'F_s^i$. Here we only assume the $c_j'$ to be rational.  When we take $K=K_{\min}$ we may write $K_{\max}=K_ {\min}+\sum_{j=1}^n c_jF_s^j$ with $c_i\in \mathbb N$.  Then $K'\in [K_{\min},K_{\max}]$ means that
% $$0\leq c_j' \leq c_j, \forall j \in J.$$

\begin {proof}

 The proof is by induction on $n$.  For $n=1$, condition $(3)$ is empty whilst $(2)$ just means that $c'_1$ lies in the convex hull of the pair $\{0,c_1\}$, as required.

 As in \ref {1.3}, lift the partial order on $\{c_i\}_{i \in N}$ induced by the natural order on $\mathbb N$ to a linear order $\prec$ and let $\textbf{c}$ be the resulting set with respect to $\prec$.

 Let $c_s$ be the unique maximal element of $\textbf{c}$ and set $\textbf{c}^-=\textbf{c}\setminus \{c_s\}$.  By the induction hypothesis we can assume that the assertion of the proposition has been proved with respect to the graph $\mathscr G(\textbf{c}^-)$, which we recall (\ref {3.1}) is relabelled by closing up gaps.

 Recall (\ref {3.1}) that $\mathscr G^+$ is the graph obtained from  $\mathscr G(\textbf{c}^-)$ by increasing the labels in $[s,n-1]$ by $1$.

% Now recall the construction of $\mathscr G(\textbf{c})$ given in \cite [Sect.7]{J3}, \cite [5.6.1]{JL}, or in the proof of Lemma \ref {5.2}.  In particular $\mathscr G^+$ is the graph obtained from  $\mathscr G(\textbf{c}^-)$ by increasing the labels in $[s,n-1]$ by $1$ (see \cite [7.2 (2)]{J3}).

  Consider a convex linear combination of elements $z_v:v \in V(\mathscr G(\textbf{c}^-))$ represented by the sequence $\textbf{c}':=(c_{n}',c_{n-1}',\ldots,c'_{s+1},c'_{s-1},\ldots,c'_1)$, with $(2)$ and $(3)$ being satisfied with respect to $\textbf{c}^-$.

 % The repetition of $c'_{s-1}$ is on account of the term $c'_{s-1}(r^{s-1}-r^{s})$ being modified to $c'_{s-1}(r^{s-1}-r^{s+1})$,
  The shifting of indices in the construction of $\mathscr G^+$ means that the term $c'_{s-1}(r^{s-1}-r^{s})$ is modified to $c'_{s-1}(r^{s-1}-r^{s+1})$ and that $c'_i$ is replaced by $c'_{i+1}$ when $n-1\geq i\geq s$.

  Thus $\textbf{c}'$ becomes
  $$\textbf{c}'':=(c_n',c_{n-1}', \ldots, c_{s+1}', c_{s-1}',c_{s-1}',c'_{s-2},c_{s-3}'\ldots,c_1') \eqno {(12)}$$ as the corresponding convex linear combination of elements  $z_v:v \in V(\mathscr G^+)$.

   Now $\mathscr G^-$ is obtained from $\mathscr G^+$ by an unlabelled graph isomorphism $\varphi$ which just relabels a vertex with label $s+1$ by the label $s$.  This replaces the term $c'_{s-1}(r^{s-1}-r^{s+1})$ by $c'_{s-1}(r^{s-1}-r^{s})$ and the term $c'_{s+1}(r^{s+1}-r^{s+2})$ by $c'_{s+1}(r^{s}-r^{s+2})$, leaving the remaining terms unchanged.

   %This may be described through the linear map $\theta: \sum_{j\in \hat{N}\setminus \{s\}}\mathbb Z r_j \rightarrow \sum_{j\in \hat{N}\setminus \{s+1\}}\mathbb Z r_j$, defined by $\theta(r^{s+1})=r^s$ and $\theta(r^j)=t^j$, for $j \neq s+1$.

   Recalling the notation of \ref {3.2}, it follows that $\textbf{c}'$ becomes the element
   $$\theta(\textbf{c}''):=(c_n',c_{n-1}', \ldots, c_{s+1}', c_{s+1}',c_{s-1}',c'_{s-2},c_{s-3}'\ldots,c_1')\eqno{(13)}$$ as the corresponding convex linear combination of elements  $z_v:v\in V(\mathscr G^-)$.

   However this is not quite the end of the story since we must view $\mathscr G^-$ as a subset of $\mathscr G(\textbf{c})$.  In this we recall that  $\mathscr G(\textbf{c})$ is the union of $\mathscr G^\pm$ in which a vertex in $\mathscr G^+$ with label $s+1$ is joined to $\varphi(v)$ (which has label $s+1$) by an edge with label $s$. This changes the value of the functions associated to the vertices of $\mathscr G^-$.  We calculate this change below.

   Fix $v \in V(\mathscr G^+)$, so then $\varphi(v) \in V(\mathscr G^-)$.
   %Let $\theta(z_v)$ denote the function obtained from $z_v$ through $(31),(32)$.
   Let $\delta_s$ denote the $n$-tuple which is $1$ is the $s^{th}$ entry and $0$ elsewhere.  We claim that
   $$z_{\varphi(v)}=\theta(z_v)+(c_s-c_{s+1})\delta_s. \eqno {(14)}$$

   To show this, consider neighbours $v,v' \in \mathscr G^+$. Recalling $(P_1)$ (resp. $(P_2)$) of \cite [Sect. 6]{J1}, let $i_v \in \hat{N}$ (resp.  $i_{v,v'} \in N$) be the label assigned to the vertex $v$ (resp. assigned to the edge joining $v,v'$).  Then by $P_3$ of \cite [Sect. 6]{J1} one has
   $$z_v-z_{v'}=c_{v,v'}(r^{i_v}-r^{i_{v'}}). \eqno {(15)}$$

   By definition of $\varphi$ the indices on coefficients do not change, that is $c_{\varphi(v),\varphi(v')}=c_{v,v'}$.  Again $i_v-i_{\varphi(v)}$ equals $0$ unless $i_v=s+1$, in which case it equals $1$.  Then from the definition of $\theta$ it follows from $(15)$ that
   $$z_{\varphi(v)}-z_{\varphi(v')}=\theta(z_v)-\theta(z_{v'}). \eqno {(16)}$$

   Recall ($(P_5)$ of \cite [6.3]{J1}) that $\mathscr G(\textbf{c})$ is connected. Then by $(16)$, it suffices to establish $(14)$ for just one element $v\in V(\mathscr G^+)$.  In this we shall adopt the convention that $r^{n+1}=0$.

   Recall \cite [6.3]{J1} the notion of a pointed chain and let $v \in V(\mathscr G^+)$ be the unique element in the pointed chain with label $s+1$, that is to say $i_v=s+1$.  Then either $z_v$ is the driving function $z$ or the previous element in the pointed chain has label $s+2$, in which case $z_v=z+\sum_{j=s+1}^nc_j(r^j-r^{j+1})$.  On the other hand the next element in the pointed chain has label $s$ and indeed coincides with $\varphi(v)$.

   Since $v, \varphi(v)$ are joined by an edge with label $s$ it follows from $P_3$ of \cite [6.3]{J1} that $z_v-z_{\varphi(v)}=c_s(r^{s+1}-r^s)$.  On the other hand from the previous formula we obtain $z_v-\theta(z_v)=c_{s+1}((r^{s+1}-r^{s+2})-(r^{s}-r^{s+2}))=-c_{s+1}(r^s-r^{s+1})$.  Subtraction gives $(14)$ for this particular choice of $v \in V(\mathscr G^+)$.

   In view of $(14)$ it follows that  $\textbf{c}'$ becomes the element
   $$\textbf{c}''':=(c_n',c_{n-1}', \ldots, c_{s+1}', c_{s+1}'+(c_s-c_{s+1}),c_{s-1}',c'_{s-2},c_{s-3}'\ldots,c_1')\eqno{(17)}$$
    as the corresponding convex linear combination of elements  $z_v:v\in V(\mathscr G^-)$, when $\mathscr G^-$ is viewed as a subgraph of $\mathscr G(\textbf{c})$ by the above construction.

   Taking a convex linear combination of these two elements lying in $K(\textbf{c})$ we may construct a third element whose entry $c_s'$ in the $s^{th}$ place is a convex linear combination of the pair $(c'_{s-1}, c'_{s+1}+c_s-c_{s+1})$.

   It remains to show that $c_j':j \in N$ give the most general possible solution to $(2),(3)$.  For $j\neq s$ this holds by the induction hypothesis.  Furthermore in this we may note that $c'_{s+1}-c_{s-1}' \geq \min \{0,c_{s+1}-c_{s-1}\} =  -\max\{0,(c_{s-1}-c_{s+1})\}$. On the other hand by choice of $s$, one has $\max\{0,(c_{s-1}-c_{s+1})\}\leq \max\{0,(c_{s}-c_{s+1})\}=c_s-c_{s+1}$.  Thus by $(3)'_{n-1}$ and the induction hypothesis $c'_{s+1}+c_s-c_{s+1}\geq c_{s-1}'$ and so by $(12)$ and $(17)$ we can make any choice of $c_s'$ satisfying
   $$c_{s-1}'\leq c_s'\leq c'_{s+1}+c_s-c_{s+1}. \eqno{(18)}$$.

    In view of the definition of $s$, the inequalities in $(18)$ are those of $(3)'_n$.  Moreover by the second inequality in $(2)$ for $k=s+1$, it also gives the second inequality in $(2)$ for $k=s$. Finally the first inequality in $(18)$ combined with the first inequality in $(2)$ for $k=s-1$, gives the first inequality in $(2)$ for $k=s$.  This completes the induction step.

   \end {proof}

   \textbf{Remark 1.}  It is not hard to check that this argument can be used to give a second proof of Proposition \ref {3.2}.

   \

   \textbf{Remark 2}.  A propos equation $(14)$ it is \textit{false} in general that
   $$z_{\varphi(v)}=z_v+(c_s-c_{s+1})\delta_s,$$
   though this does hold for some choices of $v \in V(\mathscr G^+)$.

  \subsection{}\label{3.6}

  \begin {lemma}  $K(\textbf{c})$ is a convex set.
  \end {lemma}

  \begin {proof}  Let $F$ be a finite index set and suppose that the $(c_{k,f})_{k \in N}$ satisfy $(2),(3)$ for all $f \in F$.  Let $a^f:f \in F$ be positive rational numbers summing to $1$.  We must show that $c_k':=\sum_{f \in F}a^fc_{k,f}:k \in N$, satisfy $(2),(3)$.  This is obvious for $(2)$.  Set $d_k=\max \{0,c_k-c_{k+1}\}$, which is $\geq 0$. Since the $c_{k,f}$ satisfy $(3)_n$ we obtain $c_{k,f}-c_{k+1,f}\leq d_k$, for all $k \in N$. Then $c'_k-c'_{k+1} = \sum_{f\in F}a^f(c_{k,f}-c_{k+1,f})\leq \sum_{f\in F}a^fd_k=d_k$, and so the $c'_k: k \in N$ satisfy $(3)_n$. Of course the general case is similar since this just corresponds to omitting successively elements of $N$.
  \end {proof}

  \textbf{Remark}.  This is actually a slightly more general result since we do not need to know for its proof that $s_1 \prec s_2 \prec \ldots \prec s_n$.

   \subsection{}\label{3.7}

   We may now give a proof of Theorem \ref {1.4}.  Through Proposition \ref {2.3} and Lemma \ref {3.6} it follows that $K(\textbf{c})$ contains the convex hull of $Z(\textbf{c})$ and hence is equal to it by Proposition \ref {3.5}.   Finally apply Corollary \ref {3.4}.

    \section{Reconstructing a Tableau from a Function}\label{4}

    The advantage of tableaux over binary fusion is that $f_\mathscr T: \mathscr T \in H(\textbf{c})$ can be read off directly, \cite [Sect. 4]{J1} or here by Lemma \ref {2.2},
    from the tableau $\mathscr T$ whereas the corresponding function $z_v: \in V(\mathscr G(\textbf{c}))$ is only given inductively by following the path in $\mathscr G(\textbf{c})$ to the driving function and using $(1)$.

\subsection{}\label{4.1}

  The aim of this section is to give a converse to Lemma \ref {2.2}, that is to say given a function $f$ which is linear in the co-ordinates $r^i-r^{i+1}:i \in N$ to find a complete tableau $\mathscr T$ such that $f_\mathscr T=f$.  In this we shall assume that the coefficients $\{c_i\}_{i \in N}$ are indeterminates in terms of which the coefficients of $f$ are expressed as linear functions. Otherwise the uniqueness of $\mathscr T$ is not assured.  Not all linear functions can be represented in this fashion (see Example 2 of \ref {4.2}) and so we must also establish the existence of $\mathscr T$.
  %there exists a complete tableau $\mathscr T \in H(\text{c})$ such that $f_\mathscr T$, satisfies the conclusion of Lemma \ref {2.2}.

  Let us first recall some further results of \cite {J1}.

 \subsection{}\label{4.2}

 Let $\mathscr T$ be a complete tableau (with $n+1$ columns).

 Set $u= ht(\mathscr T)$.  If $u$ is odd (resp. even) then the unique strongly extremal column \cite [Def. 2.3.4]{J1} of $\mathscr T$ is the leftmost (resp. rightmost) column of $\mathscr T$ of height $u$.

 Our first task is to identify the unique strongly extremal column of the prospective tableau $\mathscr T$.  This is provided by the remarkably simple lemma below.  Retain the notations and conventions of Lemma \ref {2.2}.  Express $f_\mathscr T$ as in $(4)$.  We recall that if $\phi$ is the empty tableau, then $f_\phi$ is the driving function which here we are taking to be equal to zero.

 \begin {lemma}  Suppose that $\mathscr T$ is not the empty tableau.  Then for all $k \in [0,n]$ the unique strongly extremal column of $\mathscr T$ is $C_{k+1}$ if and only if $c^\mathscr T_{k+1}-c_k^\mathscr T=c_{k+1}$, for all $k \in [0,n]$.

\end {lemma}

\begin {proof}  Let $C_{k+1}: k \in [1,n]$ be the unique strongly extremal column of $\mathscr T$.  It has height $u$ which by hypothesis is $>0$.  If $u$ is odd (resp. even) then $\iota^-_{C_{k+1}}=0$ (resp. $\iota^+_{C_{k+1}}=0$).  Thus only if follows from definitions and the formulae in (i),(ii) of Lemma \ref {2.2}.

%Consider now ``only if''.

Suppose that $c^\mathscr T_{k+1}-c^\mathscr T_k=c_{k+1}$.  Again suppose that $u:=ht(\mathscr T)$ is odd (resp. even).  Then by Lemma \ref {2.2} it follows that $ht(C_{k+1})$, has odd (resp. even) height and that $\iota^-_{C_{k+1}}=0$ (resp. $\iota^-_{C_{k+1}}=0$).

It remains to show that $ht(C_{k+1})=u$. Suppose $ht(C_{k+1})<u$.   Take $u$ odd (resp. even) and recall \cite [Lemma 2.3.2]{J1} that $ht(C_1)\geq u-1$ (resp. $ht(C_{n+1}) \geq u-1$).  By our hypothesis,  equality must hold. Thus either $C_{k+1}=C_1$ and has even height, or $C_{k+1}=C_{n+1}$ and has odd height.
Recalling that $c_{n+1}^\mathscr T,c_{n+1},c_0^\mathscr T,c_0$ are all zero by definition, this means that $k=0$ and $c_1^\mathscr T=c_1$ or $k=n$ and $c_n^\mathscr T=0$.

On the other hand since $ht(C_{k+1})<u$, it follows that $j \in [k+1,n+1]$ (resp. $j \in [1,k]$) defined in (ii) (resp. (i)) of Lemma \ref {2.2} exists. Then by (ii) (resp. (i)) of Lemma \ref {2.2} we obtain $c_1^\mathscr T=c_1-c_j$ (resp. $c_n^\mathscr T=c_j$), which is a contradiction.
\end {proof}

\textbf{Example 1.} Take $n=3$ and let $\mathscr T$ be defined by the partition $(3,2,1,3)$.  Then $f_\mathscr T=c_1(r^1-r^4)+(c_2-c_3)(r^2-r^3)$, and so $c^\mathscr T_1-c^\mathscr T_0=c_1, c^\mathscr T_2-c^\mathscr T_1=c_2-c_3,c^\mathscr T_3-c^\mathscr T_2=c_3-c_2,c^\mathscr T_4-c^\mathscr T_3=-c_1$.  Then the lemma $C_1$, is the strongly extremal column and indeed this is the case.

\

\textbf{Example 2.} Take $n=2$ and $f=c_1'(r^1-r^2)+c_2'(r^2-r^3)$, with $c_1'=c_1,c_2'-c_1'=c_2$. Then a tableau $\mathscr T$ for which $f_\mathscr T=f$ would have to have both $C_1$ and $C_2$ as its unique strongly extremal column.  In other words one \textit{cannot} represent the function $f=c_1(r^1-r^2)+(c_1+c_2)(r^2-r^3)$ through a tableau.  Of course if we take $c_1+c_2$ to define a new coefficient $c_2$, then this function is represented by the tableau defined by the partition $(1,1,1)$.

 \subsection{}\label{4.3}

 The above result seems at first sight to be in conflict with the rather important \cite [Lemma 5.4]{J1}.  The latter asserts that when we add to $f_\mathscr T$ the driving function $f_\phi:=-\sum_{j=1}^nc_jm^j$ (which we had previously taken to be zero) and replace $(r^j-r^{j+1})$ by $m^j+m^{j+1}$, then the coefficient of $m^{k+1}$ equals zero when $C_{k+1}$ is the unique strongly extremal column of $\mathscr T$.  Then from
 $$f_\mathscr T= \sum_{j=1}^nc_j^\mathscr T(m^j+m^{j+1}) - c_jm^j, $$
we obtain $c^\mathscr T_{k+1}+c^\mathscr T_k-c_{k+1}=0$, whilst from Lemma \ref {4.2} we obtain $c^\mathscr T_{k+1}-c^\mathscr T_k-c_{k+1}=0$.  Together \cite [Lemma 5.4]{J1} and Lemma \ref {4.2} imply the

 \begin {cor} Take $k \in [0,n]$.  If $C_{k+1}$ is the unique strongly extremal column of $\mathscr T$ then $c^\mathscr T_k=0$.
 \end {cor}

 \begin {proof}  To be reassured we give a direct proof. We can assume that $\mathscr T$ is complete, since completion does not change the strongly extremal column \cite [2.4]{J1}, nor by \cite [4.5]{J1} the associated function $f_\mathscr T$.  One has $u:=ht(\mathscr T)=ht(C_{k+1})$.  If $u$ is odd, $C_{k+1}$ is a left extremal column and so the contribution to the coefficient of $r^k-r^{k+1}$ from the row of height $u$ is zero.  On the other hand by \cite [Lemmas 2.3.3, 2.3.2] {J1} if $u$ is odd, the height of $C_k$ (resp. $C_1$) is $\geq u-2$ (resp. $\geq u-1$) and if $u$ is even, the height of $C_k$ (resp. $C_1$) $\geq u-1$ (resp. $u$).  Then using \cite [4.4]{J1} and the fact that $\mathscr T$ is well-numbered \cite [3.1]{J1}, one checks the contributions to $r^k-r^{k+1}$, from the remaining rows cancel, in pairs.  Thus $c_k^\mathscr T=0$.
 \end {proof}

 \textbf{Remark}.  Conversely this gives a further proof of \cite [Lemma 5.4]{J1}, which was an important result in the application of $S$-graphs to the required invariance property of dual Kashiwara functions.

 \

 \textbf{Example}.  Take $n=2$ and consider the tableau $\mathscr T$ defined by the partition $(2,1,2)$.  Then $C_3$ is its unique strongly extremal column, whilst $f_\mathscr T=(c_1-c_2)(r^1-r^2)$, and indeed $c^\mathscr T_2=0$.

  \subsection{}\label{4.4}

  Let $\mathscr T$ be a non-empty complete tableau of height $u$ and let $C_{k+1}$ be its unique strongly extremal column.

  A column $C_{j+1}$ of $\mathscr T$ of height $u$ is said to be quasi-extremal \cite [3.3.1]{J1} if it is not $C_{k+1}$.  By definition of the strongly extremal column, if $u$ is odd (resp. even), a quasi-extremal column $C_{j+1}$ of height $u$ must lie to the right (resp. left) of $C_{k+1}$, that is $j>k$ (resp. $j<k$).

  A tableau $\mathscr T$ does not possess a quasi-extremal column of height $u:=ht(\mathscr T)$, if and only if its strongly extremal column $C_{k+1}$ is the unique column of height $ht(\mathscr T)$.  In this case the block in row $u$ may be removed from the strongly extremal column to obtain a tableau $\mathscr T'$ and moreover $f_{\mathscr T'}=f_\mathscr T$ (as follows trivially from \cite [4.4]{J1}.  Moreover since $\mathscr T$ is assumed complete, the tableau so obtained from $\mathscr T$ is still complete.  Indeed completion adjoins dominoes by the rule described in \cite [2.3.1]{J1}).  If one cannot adjoin a domino to $\mathscr T$, then a fortiori one cannot adjoin a domino to $\mathscr T'$.  One may further remark that by the boundary conditions \cite [2.2]{J1} it follows that $C_{k+1}$ remains the strongly extremal column of $\mathscr T$.

  Thus without loss of generality we may assume that $\mathscr T$ has at least two columns of height $u$, or that $\mathscr T$ is empty.

  Let $C_{k+1}$ be the unique strongly extremal column of $\mathscr T$ and let $C_{j+1}$ be a second column of height $u$, which we can assume is a neighbour at level $u$.
 Then $j-k>0$ (resp $k-j>0$) if $u$ is odd (resp. even) and is minimal with the property that $C_{j+1}$ has height $u$. This determines $j \in [0,n]$ uniquely.

% Moreover in this case the block in row $u$ may be removed from $C_{k+1}$ to obtain a complete tableau $\mathscr T'$. Moreover the latter admits $C_{j+1}$ as its unique strongly extremal column.
%  %of $\mathscr T'$ and conversely $C_{j+1}$ is the unique neighbour of $C_{k+1}$ at level $u$.
% We claim that
%
% $$ f_\mathscr T-f_{\mathscr T'}=\left\{
%                       \begin{array}{ll}
%                         c_{k+1}(r^{k+1}-r^{j+1}), & \hbox{if $u$ is odd (noting that $k<j$),} \\
%                         c_k(r^{k+1}-r^{j+1}), & \hbox{if $u$ is even (noting that $k >j$).}
%                       \end{array}
%                     \right. \eqno {(19)}$$
%
%This follows easily from \cite [4.4]{J1} recalling that a complete tableau is well-numbered \cite [3.1]{J1}.

 Our second task is to identify the unique quasi-extremal column which is a neighbour at level $u$ of the unique strongly extremal column in the prospective tableau $\mathscr T$ of height $u$.  This is provided by the

\begin {lemma} Let $\mathscr T$ be a complete tableau of height $u>0$ with at least two columns of height $u$.  Let $C_{k+1}$ be the unique strongly extremal column of $\mathscr T$.  Then $C_{j+1}:j \in [1,n]\setminus \{k\}$ is the unique neighbour of $C_{k+1}$ at level $u$ if and only if $k< j$ and $c_{j+1}^\mathscr T-c_j^\mathscr T=c_{j+1}-c_{k+1}$ (in which case $u$ is odd) or $k>j$ and $c_{j+1}^\mathscr T-c_j^\mathscr T=c_{j+1}-c_k$ (in which case $u$ is even).
\end {lemma}

\begin {proof} By a suitable relabelling, only if follows from definitions and the formulae in (i) and (ii) of Lemma \ref {2.2}.

Consider the converse.

Take $k$ equal to $j$ in Lemma \ref {2.2}.  Then by its conclusion we can write $c^\mathscr T_{j+1}-c^\mathscr T_j=c_{j+1}-c$, where $c=0$ or is some $c_{k'}:k' \in [1,n]$.  Comparison with the hypothesis of the lemma (remembering that the $\{c_i\}_{i\in N}$ are supposed indeterminates) excludes the case $c=0$ and forces $k'=k+1$, if $k<j$ and $k'=k$ if $k>j$.  From the way in which $k'$ is obtained in applying Lemma \ref {2.2}, this implies in the first (resp. second) case that $k'\leq j$ (resp. $k'>j$) and $C_{k+1}$ is a left (resp. right) neighbour to $C_{j+1}$ at level $ht (C_{j+1})$.

In the first case above $u$ is odd and so $k<n$ (by \cite [2.3.4]{J1}) and since $\mathscr T$ is complete, $ht(C_{k+2})\geq u-1$ (by \cite [2.3.3]{J1}).  Thus $ht(C_{j+1})\geq u-1$.  If equality holds, then $ht(C_{j+1})$ is even and by Lemma \ref {2.2}(ii),(iii) we obtain a second and different formula for $c^\mathscr T_{j+1}-c^\mathscr T_j-c_{j+1}$, which is contradictory.

The proof in the second case is similar and also obtains from the first case by duality.  In more detail $u$ is even and so $k>0$ (by \cite [2.3.4]{J1}). Thus $ht(C_k)\geq u-1$ (by \cite [2.3.3]{J1}).  Thus $ht(C_{j+1})\geq u-1$.  If equality holds,  then $ht(C_{j+1})$ is odd and by Lemma \ref {2.2}(i) we obtain a second and different formula for $c^\mathscr T_{j+1}-c^\mathscr T_j-c_{j+1}$, which is contradictory.

 %Set $z_v=f_\mathscr T, z_{v'}=f_{\mathscr T'}$.

%Suppose $k<j$ and $c_{j+1}^\mathscr T-c_j^\mathscr T=c_{j+1}-c_{k+1}$.  By $(19)$ it follows that
%
%$$ (c^\mathscr T_{\ell+1}-c^\mathscr T_\ell)-(c^{\mathscr T'}_{\ell+1}-c^{\mathscr T'}_\ell)=\left\{
%                       \begin{array}{ll}
%                        -c_{k+1}, & \hbox{if $\ell=j$,} \\
%                        c_{k+1}, & \hbox{if $\ell=k$,}\\
%                        0,& \hbox{otherwise}.
%                       \end{array} \right. $$
%
%Consequently $c^\mathscr T_{j+1}-c^\mathscr T_j=c_{j+1}$, which by Lemma \ref {4.2} means that $C_{j+1}$ is the unique strongly extremal column of $\mathscr T'$ and hence that $C_{j+1}$ is the unique neighbour of $C_{k+1}$ at level $u$.
%
%Suppose $j<k$ and $c_{j+1}^\mathscr T-c_j^\mathscr T=c_{j+1}-c_{k}$.  By $(19)$ it follows that
%
%$$ (c^\mathscr T_{\ell+1}-c^\mathscr T_\ell)-(c^{\mathscr T'}_{\ell+1}-c^{\mathscr T'}_\ell)=\left\{
%                       \begin{array}{ll}
%                        -c_{k}, & \hbox{if $\ell=k$,} \\
%                        c_{k}, & \hbox{if $\ell=j$,}\\
%                        0,& \hbox{otherwise}.
%                       \end{array} \right. $$
%
%Consequently $c^\mathscr T_{j+1}-c^\mathscr T_j=c_{j+1}$, which by Lemma \ref {4.2} means that $C_{j+1}$ is the unique strongly extremal column of $\mathscr T'$ and hence that $C_{j+1}$ is the unique neighbour of $C_{k+1}$ at level $u$.
%

\end {proof}

\subsection{}\label{4.5}

Take a function $f$ of the form
$$f=\sum_{j=1}^nc_j'(r^j-r^{j+1}),$$
where the $c_j': j \in N$ are linear combinations of the $\{c_i\}_{i\in N}$ viewed as indeterminates. We can assume that $f$ is non-zero, otherwise $f=f_\phi$, where $\phi$ is the empty tableau.

Suppose that there is a complete tableau $\mathscr T \in H^{n+1}$ with $f_\mathscr T=f$.  Then by Lemma \ref {4.2}, there must be a unique $k \in [0,n]$ such that $c_{k+1}'-c_k'=c_{k+1}$. This determines the the unique strongly extremal column of $\mathscr T$ to be $C_{k+1}$.  Then (as noted in \ref {4.4}) we can assume that there is a second column of $\mathscr T$ with the same height $u$ as $C_{k+1}$ and hence a unique column $C_{j+1}:j \in [1,n]\setminus \{k\}$ of height $ht(\mathscr T)$ which is a neighbour to $C_{k+1}$ (on the appropriate side) at level $u$.

Then by Lemma \ref {4.4},  there must be a unique $j \in [1,n]\setminus \{k\}$ such that $c_{j+1}'-c_j'$ satisfies the conditions given in its conclusion.  This determines $j\in [1,n]\setminus \{k\}$ defined in the previous paragraph. Then as noted in \ref {4.4} one may remove the block at level $u$ in $C_{k+1}$ to obtain a new complete tableau $\mathscr T'$ with $C_{j+1}$ as its unique strongly extremal column. The latter admits $C_{j+1}$ as its unique strongly extremal column.
  %of $\mathscr T'$ and conversely $C_{j+1}$ is the unique neighbour of $C_{k+1}$ at level $u$.

 We claim that

 $$ f_\mathscr T-f_{\mathscr T'}=\left\{
                       \begin{array}{ll}
                         c_{k+1}(r^{k+1}-r^{j+1}), & \hbox{if $u$ is odd (noting that $k<j$),} \\
                         c_k(r^{k+1}-r^{j+1}), & \hbox{if $u$ is even (noting that $k >j$).}
                       \end{array}
                     \right. \eqno {(19)}$$

This follows easily from \cite [4.4]{J1} recalling that a complete tableau is well-numbered \cite [3.1]{J1}.

Suppose $k<j$ and $c_{j+1}^\mathscr T-c_j^\mathscr T=c_{j+1}-c_{k+1}$.  By $(19)$ it follows that

$$ (c^\mathscr T_{\ell+1}-c^\mathscr T_\ell)-(c^{\mathscr T'}_{\ell+1}-c^{\mathscr T'}_\ell)=\left\{
                       \begin{array}{ll}
                        -c_{k+1}, & \hbox{if $\ell=j$,} \\
                        c_{k+1}, & \hbox{if $\ell=k$,}\\
                        0,& \hbox{otherwise}.
                       \end{array} \right. $$

Consequently $c^{\mathscr T'}_{\ell+1}-c^{\mathscr T'}_\ell=c_{\ell+1}$, exactly when $\ell =j$.  By Lemma \ref {4.2} this is compatible with $C_{j+1}$ being the unique strongly extremal column of $\mathscr T'$.
% and hence that $C_{j+1}$ is the unique neighbour of $C_{k+1}$ at level $u$.

Suppose $j<k$ and $c_{j+1}^\mathscr T-c_j^\mathscr T=c_{j+1}-c_{k}$.  By $(19)$ it follows that

$$ (c^\mathscr T_{\ell+1}-c^\mathscr T_\ell)-(c^{\mathscr T'}_{\ell+1}-c^{\mathscr T'}_\ell)=\left\{
                       \begin{array}{ll}
                        -c_{k}, & \hbox{if $\ell=k$,} \\
                        c_{k}, & \hbox{if $\ell=j$,}\\
                        0,& \hbox{otherwise}.
                       \end{array} \right. $$

Consequently $c^{\mathscr T'}_{\ell+1}-c^{\mathscr T'}_\ell=c_{j\ell+1}$, exactly when $\ell=j$.  By Lemma \ref {4.2} this is compatible with  $C_{j+1}$ being the unique strongly extremal column of $\mathscr T'$.
%and hence that $C_{j+1}$ is the unique neighbour of $C_{k+1}$ at level $u$.

%We conclude that the condition imposed by Lemma \ref {4.2} on $\mathscr T'$ is automatically satisfied.   (It is equivalent to the condition imposed on $\mathscr T$ by Lemma \ref {4.4}.)
We thus obtain a new function $f'$ with coefficients determined from those of the original function $f$ with $f-f'$ given by the right hand side of $(19)$. The previous computation shows that the condition imposed by Lemma \ref {4.2} on $f'$ is automatically satisfied.   (Given that $f'$ is defined, it is equivalent to the condition imposed on $f$ by Lemma \ref {4.4}.)

Finally we continue inductively, that is to say by imposing the condition on $f'$ implied by Lemma \ref {4.4} and so on.  Eventually the empty tableau is reached which must correspond to the zero function.

It is clear that this gives a sequence of conditions on $f$ which \textit{if satisfied}, determines a complete tableau $\mathscr T$ such that $f_\mathscr T=f$.   This is what we wished to exhibit.

\end{document}